\documentclass[12pt]{amsart}
\usepackage{amscd,amssymb}
\usepackage{url}
\newtheorem{theorem}{Theorem}

\newtheorem{corollary}[theorem]{Corollary}
\theoremstyle{definition}

\theoremstyle{remark}
\newtheorem{remark}[theorem]{Remark}

\def\ra{{\rightarrow}}

\def\lra{{\longrightarrow}}
\def\om{{\omega}}

\def\La{{\Lambda}}

\def\bC{{\mathbb C}}

\def\bR{{\mathbb R}}

\def\Sg{{\Sigma_g}}

\def\M1g{{\mathcal M}_{g,1}}
\def\I1g{{\mathcal I}_{g,1}}

\catcode`\@=\active 

\begin{document}

\title[The Gel'fand-Kalinin-Fuks class]
{The Gel'fand-Kalinin-Fuks class and characteristic classes of transversely symplectic foliations}
\author{D.~Kotschick}
\address{Mathematisches Institut, Ludwig-Maxi\-mi\-lians-Universit\"at M\"unchen,
Theresienstr.~39, 80333 M\"unchen, Germany}
\email{dieter{\char'100}member.ams.org}
\author{S.~Morita}
\address{Department of Mathematical Sciences\\
University of Tokyo \\Komaba, Tokyo 153-8914\\
Japan}
\email{morita{\char'100}ms.u-tokyo.ac.jp}

\keywords{Gel'fand-Fuks cohomology, Hamiltonian vector fields, 
transversely symplectic foliation, exotic characteristic class, graph cohomology}
\thanks{The second author is partially supported by Grant-in-Aid for Scientific
Research 19204003, Japan Society for Promotion of Science
and the Global COE program at Graduate School of Mathematical Sciences, the University of Tokyo}

\subjclass[2000]{Primary 57R32, 57R17, 57R20; Secondary 17B66, 57R50, 58H10}
\date{October 18, 2009}

\begin{abstract}
In the early 1970's, Gel'fand, Kalinin and Fuks \cite{GKF} found an 
exotic characteristic class of degree $7$ in the Gel'fand-Fuks cohomology 
of the Lie algebra of formal Hamiltonian vector fields on the plane. We prove that this
cohomology class can be decomposed as a product $\eta\land\omega$ of a
certain leaf cohomology class $\eta$ of degree $5$ and the transverse symplectic
class $\omega$. This is similar to the well known factorization of the Godbillon-Vey 
class for codimension $n$ foliations \cite{Ghys}, \cite{Kot}. We also interpret
the characteristic classes of transversely symplectic foliations 
introduced by Kontsevich \cite{K} in terms of the known
classes and prove non-triviality for some of them.
\end{abstract}

\maketitle

\section{Introduction}

Let $\mathfrak{a}_n$ denote the Lie algebra consisting of all the formal 
vector fields on $\bR^n$ and let $\mathfrak{ham}_{2n}\subset \mathfrak{a}_{2n}$
denote the subalgebra of {\it Hamiltonian} formal vector fields
on $\bR^{2n}$ with respect to the standard symplectic form $\omega$.
Hereafter we denote by $H^{2n}_\bR$ this standard symplectic vector space, 
which is the fundamental representation of the symplectic group
$\mathrm{Sp}(2n,\bR)$.

The Lie algebra $\mathfrak{ham}_{2n}$ contains the Lie subalgebra $\mathfrak{sp}(2n,\bR)$ 
consisting of linear Hamiltonian vector fields.
In \cite{GKF} Gel'fand, Kalinin and Fuks studied the Gel'fand-Fuks
cohomology of $\mathfrak{ham}_{2n}$ and showed that
$$
H^*_{GF}(\mathfrak{ham}_{2n}, \mathfrak{sp}(2n,\bR))_{\leq 0}
\cong \bR[\omega, p_1,\cdots, p_{n}]/I \ .
$$
Throughout this paper, all the cohomology groups of Lie algebras
are with trivial coefficients in $\bR$, and we omit the coefficients from the notation.
In the formula above $\omega\in H^2_{GF}(\mathfrak{ham}_{2}, \mathfrak{sp}(2n,\bR))_{-2n}$
denotes the symplectic class of weight $-2$
(for the definition of {\it weights}, see the next section) and 
$p_i\in H^{4i}_{GF}(\mathfrak{ham}_{2n}, \mathfrak{sp}(2n,\bR))_{0}\ (i=1,\cdots,n)$ 
denote the Pontrjagin classes.
Further $I$ is the ideal generated by the classes
$$
\omega^k p_1^{k_1}\cdots p_n^{k_n} \quad (k+k_1+2k_2\cdots +n k_n > n) 
$$
that vanish. (In the context of $\mathfrak{a}_n$ this corresponds to the Bott vanishing theorem.) 
Thus, in the non-positive weight
part, the result is similar to the case of $\mathfrak{a}_n$. However, in the positive
weight part, Gel'fand, Kalinin and Fuks  found an exotic class 
$GKL \in H^7_{GF}(\mathfrak{ham}_{2}, \mathfrak{sp}(2,\bR))_{8}$ of 
degree $7$ and weight $8$,
which is now called the Gel'fand-Kalinin-Fuks class. They raised the problem
of determining whether their class is non-trivial as a characteristic class
of transversely symplectic foliations, or not. Recall that the Godbillon-Vey class,
which corresponds to 
$h_1c_1^{n}\in H^{2n+1}_{GF}(\mathfrak{a}_n, \mathrm{O}(n))$, was shown to be
non-trivial almost immediately after its discovery. Namely Roussarie
first proved the non-triviality and Thurston \cite{T} proved 
the remarkable result that this class can vary continuously.
In sharp contrast with this, non-triviality of the GKF-class has now been an open problem for nearly 40 years.

In the late 1990's, Kontsevich \cite{K} interpreted the Rozansky-Witten invariants
in terms of the Gel'fand-Fuks cohomology and characteristic classes for foliations.
As an application, he constructed certain characteristic classes for
transversely symplectic foliations. More precisely, he considered the two Lie subalgebras
$$
\mathfrak{ham}_{2n}^1\subset\mathfrak{ham}_{2n}^0\subset \mathfrak{ham}_{2n}
$$ 
of $\mathfrak{ham}_{2n}$,
where $\mathfrak{ham}_{2n}^\epsilon$ denotes the formal Hamiltonian
vector fields {\it without constant terms} and
{\it without constant or linear terms} for $\epsilon = 0, 1$
respectively. Kontsevich constructed a natural homomorphism
\begin{align*}
\land \omega^n: H^*_{GF}(\mathfrak{ham}_{2n}^0,\mathrm{Sp}(2n,\bR))&
\cong H^*_{GF}(\mathfrak{ham}_{2n}^1;\bR)^{\mathrm{Sp}(2n,\bR)}\\
&\lra H^{*+2n}_{GF}(\mathfrak{ham}_{2n},\mathrm{Sp}(2n,\bR)).
\end{align*}
Since the abelianization of the Lie algebra $\mathfrak{ham}_{2n}^1$
can be written as
$$
\mathfrak{ham}_{2n}^1\lra S^3 H^{2n}_\bR \ ,
$$
where $S^3 H^{2n}_\bR$ denotes the third symmetric power of $H^{2n}_\bR$,
he obtained a homomorphism
\begin{equation}
\Phi: H^*(S^3 H^{2n}_\bR)^{\mathrm{Sp}(2n,\bR)}\lra  
H^{*+2n}_{GF}(\mathfrak{ham}_{2n},\mathrm{Sp}(2n,\bR)).
\label{eq:K}
\end{equation}

Roughly speaking, Kontsevich first considered the {\it leaf} or {\it foliated} cohomology classes
of transversely symplectic foliations, rather than the de Rham
cohomology, and then produced characteristic classes for such foliations (in de Rham 
cohomology) by taking the wedge product
with the maximal power $\omega^n$ of the transverse symplectic form.

The purpose of this paper is twofold.
Firstly we interpret the Gel'fand-Kalinin-Fuks
class in this framework of Kontsevich. 
This interpretation shows that the GKF
class can be decomposed as a product $\eta\land\omega$ of a
certain leaf cohomology class $\eta$ of degree $5$ and the transverse symplectic
class $\omega$. 
This is similar to the case of the Godbillon-Vey class $h_1c_1^n$
for codimension $n$ foliations \cite{Ghys}, which can be expressed as
the product of a $1$-dimensional leaf cohomology class $h_1$,
the Reeb class, and the primary characteristic class $c_1^n$. 
(Similar factorizations are known for some other characteristic classes of 
foliations; see e.g. \cite{Kot}.)
Although the problem of
geometric non-triviality of the GKF class remains open,
we hope that our result will shed some light on 
the geometric meaning of this class.
Secondly we determine Kontsevich's homomorphism $\Phi$ 
in \eqref{eq:K}
completely up to degree $2n$
and prove that some of these classes are 
non-trivial as characteristic classes of
transversely symplectic foliations.


\section{Gel'fand-Fuks cohomology of formal Hamiltonian vector fields}

As is well-known, each element $X\in \mathfrak{ham}_{2n}$ corresponds
bijectively to a formal Hamiltonian function 
$$
H\in \bR[[x_1,\cdots,x_n,y_1,\cdots,y_n]]/\bR
$$
which is defined up to constants, via the correspondence
$$
X\leftrightarrow \sum_{i=1}^n 
\left\{\frac{\partial H}{\partial x_i}\frac{\partial}{\partial y_i}-
\frac{\partial H}{\partial y_i}\frac{\partial}{\partial x_i}\right\}.
$$
Thus, on the one hand, we have an isomorphism of topological Lie algebras
$$
\mathfrak{ham}_{2n}\cong \bR[[x_1,\cdots,x_n,y_1,\cdots,y_n]]/\bR \ ,
$$
where the Lie bracket on the right hand side is given by the Poisson bracket.
On the other hand, this topological Lie algebra is the completion
of that of polynomial Hamiltonian functions
$$
\bR[x_1,\cdots,x_n,y_1,\cdots,y_n]/\bR
=\bigoplus_{k=1}^\infty S^k H_\bR^{2n}.
$$
Here $S^k H_\bR^{2n}$ denotes the $k$-th symmetric power of 
$H_\bR^{2n}$, which is identified with the space of all the
homogeneous polynomials of degree $k$.

The Poisson bracket is given
by
$$
S^k H_\bR^{2n}\otimes S^\ell H_\bR^{2n}\ni f\otimes g
\mapsto \{f,g\}\in S^{k+\ell-2} H_\bR^{2n}.
$$
Hence the cochain complex $C^*_{GF}(\mathfrak{ham}_{2n})$ of 
$\mathfrak{ham}_{2n}$ splits as a direct sum of finite dimensional
subcomplexes
$$
C^*_{GF}(\mathfrak{ham}_{2n})
\cong 
\bigoplus_{w=-2n}^\infty
C^*_{GF}(\mathfrak{ham}_{2n})_{w} \ ,
$$
so that we have
$$
H^*_{GF}(\mathfrak{ham}_{2n})
\cong 
\bigoplus_{w=-2n}^\infty
H^*_{GF}(\mathfrak{ham}_{2n})_{w} \ .
$$
Here
$$
C^*_{GF}(\mathfrak{ham}_{2n})_{w}=
\sum_{-k_1+k_3+2 k_4+3k_5\cdots=w}
\Lambda^{k_1} (S^{1}H_\bR^{2n})^*\otimes
\Lambda^{k_2} (S^{2}H_\bR^{2n})^*\otimes\cdots
$$
denotes the set of cochains with weight $w$, where 
we define the weight of each element in $ (S^{k}H_\bR^{2n})^*$ 
to be $k-2$, so that the coboundary operator preserves the weights.
Similar decompositions hold in the case of the relative 
cohomology $H^*_{GF}(\mathfrak{ham}_{2n},\mathfrak{sp}(2n,\bR))$,
and in the cases of the Lie subalgebras
$\mathfrak{ham}^0_{2n}, \mathfrak{ham}^1_{2n}$.

Now, as was already mentioned in the introduction, Gel'fand,
Kalinin and Fuks proved in \cite{GKF} that the cohomology 
$
H^*_{GF}(\mathfrak{ham}_{2n},\mathrm{Sp}(2n,\bR))_{\leq 0}
$
in the {\it non-positive} weight part is described in terms
of the usual characteristic classes, namely the Pontrjagin classes
and the transverse symplectic class. However,
contrary to their initial working hypothesis, they found an
exotic class for the case $n=1$:
\begin{theorem}[{\bf Gel'fand-Kalinin-Fuks \cite{GKF}}]
The relative cohomology $H^*_{GF}(\mathfrak{ham}_2,\mathfrak{sp}(2,\bR))_w$
for $w\leq 8$ is given by:
\begin{align*}
H^*_{GF}(\mathfrak{ham}_2,\mathfrak{sp}(2,\bR))_{\leq 0}
&\cong \bR[\omega, p_1]/(\omega^2, \omega p_1, p_1^2) \ ,\\
H^*_{GF}(\mathfrak{ham}_2,\mathfrak{sp}(2,\bR))_{w}
&= 0\quad (w=1,\cdots,7) \ ,\\
H^*_{GF}(\mathfrak{ham}_2,\mathfrak{sp}(2,\bR))_{8}
&=
\begin{cases}
\bR \quad (*=7)\\
0\quad \text{(otherwise)} \ .
\end{cases}
\end{align*}
\end{theorem}
Perchik \cite{P} gave a formula for the generating function
$$
\sum_{w=0}^\infty \chi(H^*(\mathfrak{ham}_{2n},\mathrm{Sp}(2n,\bR))_w)t^w
$$
of the Euler characteristic of the relative cohomology of
$\mathfrak{ham}_{2n}$.
By computing it for the case $n=1$, he showed the
{\it existence} of many more exotic classes. Later, Metoki \cite{Metoki}
found an explicit exotic class in 
$H^9(\mathfrak{ham}_2,\mathfrak{sp}(2,\bR))_{14}$
which we call the Metoki class.

Similar to the case of $H^*_{GF}(\mathfrak{a}_n,\mathrm{O}(n))$,
which provides characteristic classes for foliations of
codimension $n$ (see \cite{BR}\cite{BH}),
the relative cohomology 
$H^*_{GF}(\mathfrak{ham}_{2n},\mathrm{Sp}(2n,\bR))$
provides characteristic classes for 
{\it transversely symplectic} foliations of
codimension $2n$. More precisely,
let $\mathrm{B\Gamma}_{2n}^\omega$ denote the Haefliger classifying 
space for transversely symplectic foliations of codimension $2n$.
Then we have a homomorphism
\begin{equation}
H^*_{GF}(\mathfrak{ham}_{2n},\mathrm{Sp}(2n,\bR))\lra
H^*_{GF}(\mathfrak{ham}_{2n},\mathrm{U}(n))
\lra H^*(\mathrm{B\Gamma}_{2n}^\omega;\bR) \ ,
\label{eq:sf}
\end{equation}
where $\mathrm{U}(n)\subset \mathrm{Sp}(2n,\bR)$ is a
maximal compact subgroup. In particular, we have the important
problem of determining whether the Gel'fand-Kalinin-Fuks class
in $H^7_{GF}(\mathfrak{ham}_2,\mathrm{Sp}(2,\bR))$ 
defines a non-trivial characteristic class in 
$H^7(\mathrm{B\Gamma}_2^\omega;\bR)$, or not. We also have this
problem for the Metoki class.

In \cite{K}, Kontsevich proposed a new approach to the theory of 
characteristic classes for transversely symplectic foliations.
Here we briefly summarize his construction.
Kontsevich considered the two Lie subalgebras
$$
\mathfrak{ham}_{2n}^0\supset \mathfrak{ham}_{2n}^1
$$ 
of $\mathfrak{ham}_{2n}$ 
consisting of formal Hamiltonian vector fields {\it without constant terms} and
{\it without constant or linear terms}.
In terms of Hamiltonian functions, we can write
\begin{align*}
\mathfrak{ham}_{2n}^0 &=
\left(\bigoplus_{k=2}^\infty S^k H_\bR^{2n}\right)^{\wedge}\\
\mathfrak{ham}_{2n}^1 &=
\left(\bigoplus_{k=3}^\infty S^k H_\bR^{2n}\right)^{\wedge}
\end{align*}
and we have an isomorphism
$$
H^*_{GF}(\mathfrak{ham}_{2n}^0,\mathrm{Sp}(2n,\bR))
\cong
H^*_{GF}(\mathfrak{ham}^1_{GF})^{\mathrm{Sp}(2n,\bR)}.
$$
Let $\mathcal{F}$ be a foliation on a smooth manifold $N$
and let $T\mathcal{F}\subset TN$ be the tangent bundle of $\mathcal{F}$.
The {\it leaf} cohomology or {\it foliated} cohomology of
$\mathcal{F}$, denoted by $H^*_{\mathcal{F}}(N;\bR)$,
is defined to be the cohomology of 
$\Omega^*_{\mathcal{F}}(N)=\oplus_k \Gamma(\Lambda^k T^*\mathcal{F})$,
which is the quotient of the de Rham complex $\Omega^*N$ of $N$ by the 
ideal $I^*(\mathcal{F})$ of $\mathcal{F}$. If $\mathcal{F}$ is a transversely symplectic
foliation of codimension $2n$, then there is a transverse symplectic form
$\omega$, and the homomorphism
$$
\wedge\omega^n\colon I^*(\mathcal{F})\longrightarrow \Omega^*N
$$
vanishes identically, so that there is a well-defined homomorphism
$$
\land\omega^n: H^*_{\mathcal{F}}(N;\bR) \lra
H^{*+2n}(N;\bR) \ .
$$
Now Kontsevich pointed out that the relative cohomology 
$$
H^*_{GF}(\mathfrak{ham}_{2n}^0,\mathrm{Sp}(2n,\bR))
\cong
H^*_{GF}(\mathfrak{ham}^1_{2n};\bR)^{\mathrm{Sp}(2n,\bR)}
$$
serves as the universal model for $H^*_{\mathcal{F}}(N;\bR)$,
so that one has the following commutative diagram:
\begin{equation*}
\begin{CD}
H^*_{GF}(\mathfrak{ham}^1_{2n})^{\mathrm{Sp}(2n,\bR)}
 @>>> H^*_{\mathcal{F}}(N;\bR) \\
@V{\land\omega^n}VV @VV{\land\omega^n}V \\
H^{*+2n}_{GF}(\mathfrak{ham}_{2n},\mathrm{Sp}(2n,\bR))
 @>>> H^{*+2n}(N;\bR) \ .
\end{CD}
\end{equation*}


It is easy to show that the natural projection
$$
\mathfrak{ham}_{2n}^1\lra S^3 H^{2n}_\bR
$$
onto the lowest weight part gives the abelianization of the
Lie algebra $\mathfrak{ham}_{2n}^1$ because the
Poisson bracket
$$
S^3 H^{2n}_\bR\otimes S^k H^{2n}_\bR\lra
S^{k+1}H^{2n}_\bR\
$$
is easily seen to be surjective for any $k\geq 3$.
It follows that, as mentioned in \eqref{eq:K}, there is a homomorphism $\Phi$
defined by the following composition:
$$
H^*(S^3 H^{2n}_\bR)^{\mathrm{Sp}(2n,\bR)}\lra  
H^*_{GF}(\mathfrak{ham}^1_{2n})^{\mathrm{Sp}(2n,\bR)}
\overset{\wedge \omega^n}{\lra}
H^{*+2n}_{GF}(\mathfrak{ham}_{2n},\mathrm{Sp}(2n,\bR)) \ .
$$
Further composing $\Phi$ with the homomorphism in \eqref{eq:sf}, we obtain
\begin{equation}
\widetilde{\Phi}:
H^*(S^3 H^{2n}_\bR)^{\mathrm{Sp}(2n,\bR)}
\lra H^{*+2n}(\mathrm{B\Gamma}^\om_{2n};\bR) \ .
\label{eq:K2}
\end{equation}

For any symplectic manifold $(M,\omega)$
of dimension $2n$, let $\mathrm{Symp}^\delta(M)$ denote
the symplectomorphism group of $M$ equipped with the 
{\it discrete} topology. Then the above construction
gives rise to a homomorphism
$$
H^*(S^3 H^{2n}_\bR)^{\mathrm{Sp}(2n,\bR)}\lra 
H^{*+2n}(\mathrm{ESymp}^\delta(M))
\stackrel{\int_M}{\lra}
H^{*}(\mathrm{BSymp}^\delta(M);\bR) \ ,
$$
where $\mathrm{ESymp}^\delta(M)$ denotes the total space
of the universal foliated $M$-bundle over the classifying
space $\mathrm{BSymp}^\delta(M)$ of $\mathrm{Symp}^\delta(M)$,
and $\int_M$ is the integration over the fiber in this universal bundle.

One of the merits of the above construction of Kontsevich
is that the {\it stable} cohomology of 
$\mathfrak{ham}_{2n}$ is not interesting because by \cite{GS} 
it is just the polynomial algebra on the symplectic class
while that of $\mathfrak{ham}^0_{2n}$ is one of the three versions
of Kontsevich's graph cohomology
(see \cite{Kontsevich93}, \cite{Kontsevich94}),
more precisely the {\it commutative version},
which is very rich and still mysterious.


\section{Statements of the main results}

In this section, we state the main results of this paper.

\begin{theorem}
In the range of weights $w\leq 10$ the relative cohomology groups
$H^*_{GF}(\mathfrak{ham}^0_2,\mathfrak{sp}(2,\bR))_w$
are non-trivial only for the following three combinations of degree and weight:
\begin{align*}
H^0_{GF}(\mathfrak{ham}^0_2,\mathfrak{sp}(2,\bR))_{0}&\cong\bR\\
H^2_{GF}(\mathfrak{ham}^0_2,\mathfrak{sp}(2,\bR))_{2}&\cong\bR\\
H^5_{GF}(\mathfrak{ham}^0_2,\mathfrak{sp}(2,\bR))_{10}&\cong\bR \ .
\end{align*}
Furthermore, the following homomorphisms are all isomorphisms:
\begin{align*}
\wedge\omega: H^0_{GF}(\mathfrak{ham}^0_2,\mathfrak{sp}(2,\bR))_{0}
&\ \lra \ H^2_{GF}(\mathfrak{ham}_2,\mathfrak{sp}(2,\bR))_{-2}\cong \bR<\omega> \ ,\\
\wedge\omega: H^2_{GF}(\mathfrak{ham}^0_2,\mathfrak{sp}(2,\bR))_{2}
&\ \lra \ H^4_{GF}(\mathfrak{ham}_2,\mathfrak{sp}(2,\bR))_{0}\cong \bR<p_1> \ ,\\
\wedge\omega: H^5_{GF}(\mathfrak{ham}^0_2,\mathfrak{sp}(2,\bR))_{10}
&\ \lra \ H^7_{GF}(\mathfrak{ham}_2,\mathfrak{sp}(2,\bR))_{8}\cong\bR<\text{GKF}> \ .
\end{align*}
It follows that both the first Pontrjagin class $p_1$ and
the Gel'fand-Kalinin-Fuks class GKF can be decomposed as 
wedge products of certain leaf cohomology classes and the
transverse symplectic class $\omega$.
\label{th:main}
\end{theorem}

Combining Theorem \ref{th:main} with our earlier result in
\cite{KM03} (see also \cite{KM07}) we obtain
the following non-triviality result for the
characteristic classes defined by Kontsevich \cite{K}.
\begin{corollary}
Under the homomorphisms
\begin{align*}
H^2(S^3 H_\bR^{2};\bR)^{\mathrm{Sp}(2,\bR)}&\ \lra \
H^2(\mathrm{BSymp}^\delta(\Sigma_g);\bR)\\
H^2(S^3 H_\bR^{2};\bR)^{\mathrm{Sp}(2,\bR)}&\ \overset{\widetilde{\Phi}}{\lra} \
H^4(\mathrm{B\Gamma}_2^\omega;\bR),
\end{align*}
the generator of $H^2(S^3 H_\bR^{2};\bR)^{\mathrm{Sp}(2,\bR)}\cong\bR$
is mapped to 
$$
e_1\in H^2(\mathrm{BSymp}^\delta(\Sigma_g);\bR),
\quad  p_1\in H^4(\mathrm{B\Gamma}_2^\omega;\bR)
$$
respectively (up to non-zero constants),
where $\mathrm{Symp}(\Sg)$ denotes the symplectomorphism group of
$\Sg$ with respect to a fixed symplectic form.
It follows that both homomorphisms are non-trivial.
\label{cor:nt}
\end{corollary}

We can further generalize the above result to higher
dimensions as follows.
\begin{theorem}
In the range $*\leq 2n$, the image of the homomorphism
$$
\Phi\colon H^{*}(S^3 H_\bR^{2n};\bR)^{\mathrm{Sp}(2n,\bR)}\ \lra \
H^{*+2n}_{GF}(\mathfrak{ham}_{2n},\mathrm{Sp}(2n,\bR))
$$
introduced by Kontsevich is precisely the subspace spanned by
the classes
$$
\omega^k p_1^{k_1}\cdots p_n^{k_n} 
\quad (k+k_1+2k_2\cdots +n k_n = n)
$$
that are borderline with respect to Bott vanishing in this context.
Furthermore, the elements 
$$
\om^{n}, \om^{n-1}p_1,\cdots, p_1^{n}
$$
are mapped non-trivially under the homomorphism
$$
\widetilde{\Phi}\colon H^{*}(S^3 H_\bR^{2n};\bR)^{\mathrm{Sp}(2n,\bR)}\ \lra \
H^{*+2n}(\mathrm{B\Gamma}_{2n}^\omega;\bR)
$$ 
so that
$$
\dim \mathrm{Im} \widetilde{\Phi}\geq n+1 \ .
$$
\label{th:nt2}
\end{theorem}

\begin{remark}
It seems reasonable to conjecture that the above homomorphism
$\Phi$ is trivial in the range $*>2n$. This is true for the case
$n=1$ by Theorem \ref{th:main}. 

It is an important
problem to determine whether the classes involving the
higher Pontrjagin classes $p_i \ (i\geq 2)$ are non-trivial, or not.
\end{remark}


\section{Proofs of the main results}

In this section we write $H$ as a shorthand for $H^{2n}_\bR$.

We begin with the proof of Theorem \ref{th:main}. 
For this, notice that
\begin{align*}
C^*_{GF}(\mathfrak{ham}_2^0,\mathfrak{sp}(2,\bR))_{w}&=\\
\sum_{k_3+2 k_4+3k_5\cdots=w}&
(\Lambda^{k_3} S^{3}H^*\otimes
\Lambda^{k_4} S^{4}H^*\otimes
\Lambda^{k_5} S^{5}H^*\otimes
\cdots)^{\mathrm{Sp}(2,\bR)} \ ,\\
C^*_{GF}(\mathfrak{ham}_2,\mathfrak{sp}(2,\bR))_{w}&=\\
\sum_{-k_1+k_3+2 k_4\cdots=w}&
(\Lambda^{k_1} H^*\otimes
\Lambda^{k_3} S^{3}H^*\otimes
\Lambda^{k_4} S^{4}H^*\otimes
\cdots)^{\mathrm{Sp}(2,\bR)} \ .
\end{align*}
It is easy to see that both 
$C^*_{GF}(\mathfrak{ham}_2^0,\mathfrak{sp}(2,\bR))_{w}$ and
$C^*_{GF}(\mathfrak{ham}_2,\mathfrak{sp}(2,\bR))_{w}$
vanish if $w$ is odd. Moreover, 
for $C^*_{GF}(\mathfrak{ham}_2^0,\mathfrak{sp}(2,\bR))_{w}$ with $w=2,4,6,8$
we have the following Table \ref{tab:1}, where $\chi$ denotes the Euler characteristic.

\begin{table}[h]
\caption{}
\begin{center}
\begin{tabular}{|l|c|c|c|c|c|c|}
\noalign{\hrule height0.8pt}
\hfil $k$ & $1$ & $2$ & $3$ & $4$ & $5$ & $\chi$  \\
\hline
$\dim C_{GF}^{k}(\mathfrak{ham}_2^0,\mathfrak{sp}(2,\bR))_{2}$ 
& $0$ & $1$ & $0$ & $0$ & $0$ & $1$  \\
\hline
$\dim C_{GF}^{k}(\mathfrak{ham}_2^0,\mathfrak{sp}(2,\bR))_{4}$ 
& $0$ & $0$ & $1$ & $1$ & $0$ & $0$\\
\hline
$\dim C_{GF}^{k}(\mathfrak{ham}_2^0,\mathfrak{sp}(2,\bR))_{6}$ 
& $0$ & $1$ & $1$ & $0$ & $0$ & $0$ \\
\hline
$\dim C_{GF}^{k}(\mathfrak{ham}_2^0,\mathfrak{sp}(2,\bR))_{8}$ 
& $0$ & $0$ & $4$ & $5$ & $1$ & $0$ \\
\noalign{\hrule height0.8pt}
\end{tabular}
\end{center}
\label{tab:1}
\end{table}

Here we have used well-known facts about the representations of
$\mathrm{Sp}(2,\bR)$ such as $S^kH^*\cong S^kH$ and
$$
S^kH\otimes S^{\ell}H
\cong
S^{k+\ell}H\oplus S^{k+\ell-2}H\oplus\cdots
\oplus S^{k-\ell}H\quad (k\geq \ell) \ ,
$$
as well as various formulae for the 
irreducible decomposition of $\Lambda^m  S^kH$; see e.g. \cite{FH}.

In the weight $2$ part, we find that the homomorphism
\begin{align*}
H_{GF}^{2}(\mathfrak{ham}_2^0,\mathfrak{sp}(2,\bR))_{2}
= (\Lambda^2 &S^3 H^*)^{\mathrm{Sp}(2,\bR)}\cong\bR
\ \overset{\land\omega}{\lra}\\
H_{GF}^{4}(\mathfrak{ham}_2,\mathfrak{sp}(2,\bR))_{0}
&=(\Lambda^2 H^*\otimes
\Lambda^2 S^3 H^*)^{\mathrm{Sp}(2,\bR)}
\cong\bR<p_1>
\end{align*}
is an isomorphism because
$\omega$ is a generator
of $(\Lambda^2 H^*)^{\mathrm{Sp}(2,\bR)}\cong\bR$. 
It follows that the first Pontrjagin class
$p_1$ can be decomposed as a wedge product
$$
p_1=\gamma_1\wedge \omega
$$
of a class 
$\gamma_1\in H_{GF}^{2}(\mathfrak{ham}_2^0,\mathfrak{sp}(2,\bR))_{2}\cong \bR$
in the first leaf cohomology
with the transverse symplectic class $\omega$.
In the weight $4$ part, we find that the coboundary operator
\begin{align*}
C^3_{GF}(\mathfrak{ham}_2^0,\mathfrak{sp}(2,\bR))_{4}
&=(\Lambda^2 S^3 H^*\otimes S^4 H^*)^{\mathrm{Sp}(2,\bR)}
\cong\bR\\
&\overset{\delta}{\lra}
C^4_{GF}(\mathfrak{ham}_2^0,\mathfrak{sp}(2,\bR))_{4}
=(\Lambda^4 S^3 H^*)^{\mathrm{Sp}(2,\bR)}\cong\bR
\end{align*}
is an isomorphism. This, together with the computation shown
in Table~\ref{tab:1}, shows that 
$H^*(\mathfrak{ham}_2^0,\mathfrak{sp}(2,\bR))_{4}$ is trivial.
Similarly the weight $6$ part 
$H^*(\mathfrak{ham}_2^0,\mathfrak{sp}(2,\bR))_{6}$
is trivial because the coboundary operator
\begin{align*}
C^2_{GF}&(\mathfrak{ham}_2^0,\mathfrak{sp}(2,\bR))_{6}
=(\Lambda^2 S^5 H^*)^{\mathrm{Sp}(2,\bR)}
\cong\bR\\
&\overset{\delta}{\lra}
C^3_{GF}(\mathfrak{ham}_2^0,\mathfrak{sp}(2,\bR))_{6}
=(S^3 H^*\otimes S^4 H^*\otimes S^5 H^*
)^{\mathrm{Sp}(2,\bR)}\cong\bR
\end{align*}
can be seen to be an isomorphism.

The cochain complex 
$C^{*}_{GF}(\mathfrak{ham}_2^0,\mathfrak{sp}(2,\bR))_{8}$
for the weight $8$ part is given in Table \ref{tab:2},
where the symbols $(347), (4^26)$, for example,  stand for
\begin{align*}
(S^3 H^*\otimes S^4 H^*\otimes S^7 H^*
)^{\mathrm{Sp}(2,\bR)}&\cong\bR\\
(\Lambda^2 S^4 H^*\otimes S^6 H^*
)^{\mathrm{Sp}(2,\bR)}&\cong\bR
\end{align*}
respectively, and similarly for the other ones.

\begin{table}[h]
\caption{generators for $C_{GF}^{*}(\mathfrak{ham}_2^0,\mathfrak{sp}(2,\bR))_{8}$}
\begin{center}
\begin{tabular}{|l|c|c|}
\noalign{\hrule height0.8pt}
\hfil ${}$ & $\text{dim}$ & $\text{generators}$ \\
\hline
$C_{GF}^{3}(\mathfrak{ham}_2^0,\mathfrak{sp}(2,\bR))_{8}$  & $4$
& $(347) (356) (4^26)(45^2)$\\
\hline
$C_{GF}^{4}(\mathfrak{ham}_2^0,\mathfrak{sp}(2,\bR))_{8}$  & $5$
& $(3^246) (3^25^2)_2 (34^25)_2$\\
\hline
$C_{GF}^{5}(\mathfrak{ham}_2^0,\mathfrak{sp}(2,\bR))_{8}$  & $1$
& $(3^345)$\\
\noalign{\hrule height0.8pt}
\end{tabular}
\end{center}
\label{tab:2}
\end{table}
The subscript $2$ in the symbol $(3^25^2)_2$
means that its dimension is $2$, namely we have
$$
(\Lambda^2 S^3 H^*\otimes \Lambda^2 S^5 H^*
)^{\mathrm{Sp}(2,\bR)}\cong\bR^2 \ .
$$
A direct computation of the coboundary operators shows that
this cochain complex is acyclic.

The dimensions of the cochain complex for the weight $10$ part are
given in the first line of Table \ref{tab:3}.
In the second line, the dimensions of 
the cochain complex for the weight $8$ part of
$H^*_{GF}(\mathfrak{ham}_2,\mathfrak{sp}(2,\bR))_{8}$
are given. These were first computed by Gel'fand-Kalinin-Fuks \cite{GKF}
and were later re-computed by Metoki \cite{Metoki}.

\begin{table}[h]
\caption{}
\begin{center}
\begin{tabular}{|l|c|c|c|c|c|c|c|c|c|}
\noalign{\hrule height0.8pt}
\hfil $k$ & $1$ & $2$ & $3$ & $4$ & $5$ & $6$ & $7$ & $8$ & $\chi$ \\
\hline
$\dim C_{GF}^{k-2}(\mathfrak{ham}_2^0,\mathfrak{sp}(2,\bR))_{10}$ 
& $0$ & $0$ & $0$ & $1$ & $3$  & $9$ & $12$ & $4$ & $-1$ \\
\hline
$\dim C_{GF}^{k}(\mathfrak{ham}_2,\mathfrak{sp}(2,\bR))_{8}$  
& $0$ & $0$
& $5$ & $13$ & $17$ & $18$ & $14$ & $4$ & $-1$\\
\noalign{\hrule height0.8pt}
\end{tabular}
\end{center}
\label{tab:3}
\end{table}

As was already mentioned, Gel'fand, Kalinin and Fuks determined
$H^*_{GF}(\mathfrak{ham}_2,\mathfrak{sp}(2,\bR))_{8}$ 
by a computer calculation and found that it is $1$-dimensional,
generated by their class of degree $7$. 
Metoki re-computed this cohomology group, 
again with the aid of a computer program, and constructed an explicit (but complicated) cocycle for the
$GKF$-class. His cocycle is not divisible by $\omega$, and, so far, no
cocycle divisible by $\omega$ has been known.

Now it can be checked that the homomorphism
$$
\wedge\omega\colon C^*_{GF}(\mathfrak{ham}_2^0,\mathfrak{sp}(2,\bR))_{10}
\ \lra\
C^{*+2}_{GF}(\mathfrak{ham}_2,\mathfrak{sp}(2,\bR))_{8}
$$
induces an embedding of cochain complexes which shifts the
degree by $2$ and the weight by $-2$.
Our purpose is to prove that it induces an isomorphism in cohomology.
By an explicit computation, we determined a system of generators
for the first chain complex
$C_{GF}^{*}(\mathfrak{ham}_2^0,\mathfrak{sp}(2,\bR))_{10}$
as shown in Table \ref{tab:4}.

\begin{table}[h]
\caption{generators for $C_{GF}^{*}(\mathfrak{ham}_2^0,\mathfrak{sp}(2,\bR))_{10}$}
\begin{center}
\begin{tabular}{|l|c|c|}
\noalign{\hrule height0.8pt}
\hfil ${}$ & $\text{dim}$ & $\text{generators}$ \\
\hline
$C_{GF}^{2}(\mathfrak{ham}_2^0,\mathfrak{sp}(2,\bR))_{10}$ 
& $1$ & $(7^2)$ \\
\hline
$C_{GF}^{3}(\mathfrak{ham}_2^0,\mathfrak{sp}(2,\bR))_{10}$  & $3$
& $(358) (367) (457)$\\
\hline
$C_{GF}^{4}(\mathfrak{ham}_2^0,\mathfrak{sp}(2,\bR))_{10}$  & $9$
& $(3^248) (3^257) (34^27)(3456)_4(35^3)(4^36)$\\
\hline
$C_{GF}^{5}(\mathfrak{ham}_2^0,\mathfrak{sp}(2,\bR))_{10}$  & $12$
& $(3^347) (3^356) (3^24^26)_3(3^245^2)_4(34^35)_2(4^5)$\\
\hline
$C_{GF}^{6}(\mathfrak{ham}_2^0,\mathfrak{sp}(2,\bR))_{10}$  & $4$
& $(3^45^2) (3^34^25)_2 (3^24^4)$\\
\noalign{\hrule height0.8pt}
\end{tabular}
\end{center}
\label{tab:4}
\end{table}

In general, there are three equivalent ways of expressing elements
in $C^*_{GF}(\mathfrak{ham}_{2n},\mathrm{Sp}(2n,\bR))$.
The first one is in terms of (duals of) 
$\mathrm{Sp}(2n,\bR)$-invariant tensors of Hamiltonian functions.
The second one is by means of vertex oriented graphs
which encode ways of 
contraction of tensors of Hamiltonian functions
by applying the symplectic pairing
$H^{2n}_\bR\otimes H^{2n}_\bR\ra\bR$ along the edges.
The third one is in terms of tautological $1$-forms
$$
\delta^{i}_{j_1\cdots j_k}\in C^1_{GF}(\mathfrak{a}_{2n})
$$
(restricted to the Lie subalgebra $\mathfrak{ham}_{2n}$), defined by
$$
\delta^{i}_{j_1\cdots j_k}(X)=(-1)^k 
\frac{\partial f_i}{\partial x_{j_1}\cdots\partial x_{j_k}}(0,\cdots, 0) \ ,
$$
where
$$
X=\sum_{i=1}^{2n} f_i \frac{\partial}{\partial x_i} \in \mathfrak{a}_{2n} \ ;
$$
see e.g. \cite{B}. For example, a generator of 
$(\Lambda^2 S^3 H^*)^{\mathrm{Sp}(2,\bR)}$ 
in the first line of Table \ref{tab:1} can be given in either of the following three ways:
\begin{align*}
&\mathrm{(1)}\ (x^3\wedge y^3-3x^2y\wedge xy^2)^* \ ,\\
&\mathrm{(2)}\ \text{a graph with $2$ vertices and $3$ edges joining them} \ ,\\
&\mathrm{(3)}\  -\delta^1_{22}\wedge\delta^2_{11}-3\delta^1_{11}\wedge\delta^2_{22} \ .
\end{align*}

Metoki \cite{Metoki} gave an explicit basis for 
$C^*_{GF}(\mathfrak{ham}_2,\mathfrak{sp}(2,\bR))_{8}$
and computed the coboundary operators in terms of this basis.
Although he did not use the representation theory of $\mathrm{Sp}(2,\bR)$
explicitly, it turns out that our basis given in Table \ref{tab:4}
appears as a subbasis of his.
Therefore we can use his computation 
(which we checked by our method).
In particular, the coboundary maps
$$
C^4(\mathfrak{ham}_2^0,\mathfrak{sp}(2,\bR))_{10}
\overset{A}{\lra}
C^5(\mathfrak{ham}_2^0,\mathfrak{sp}(2,\bR))_{10}
\overset{B}{\lra}
C^6(\mathfrak{ham}_2^0,\mathfrak{sp}(2,\bR))_{10}
$$
are represented by the following
$(12,9)$-matrix
$$
\setcounter{MaxMatrixCols}{18}
A=
\begin{pmatrix}
45 & 18 & 9 & 0 & 0 & 0 & 0 & 0 & 0  \\
0 & -9 & 0 & 0 & 5 &6&6&0&0\\
-10 & 0 & -2 & 10 & -10 &-19&-33&0&1\\
-120 & 0 & 72 & 10 & -16 &-3&16&0&6\\
30 & 0 & -30 & 0 & -2 &-12&-21&0&-3\\
-8 & 50 & 0 & 10 & -32 &-48&-60&13&0\\
-1 & -2 & 0 & 10 & -4 &-15&-25&11&0\\
-15 & 18 & 0 & 20 & -34 &-57&-73&0&0\\
-70 & 16 & 0 & 0 & 0 &6&-14&20&0\\
0 & 0 & 0 -9 & -20 & 20 &39&52&0&-1\\
0 & 0 & 57 & 10 & 4 &12&16&0&-3\\
0 & 0 & 0 & 0 & 0 &0&0&0&-140
\end{pmatrix}
$$
and the $(4,12)$-matrix 
$$
\setcounter{MaxMatrixCols}{18}
B=
\begin{pmatrix}
0 & 30 & -39 & 0 \\
-56 & -8 & -12 & 0 \\
0 & 6 & 9 & 0 \\
0 & 28 & -28 & -14 \\
0 & 68 & -66 & -56 \\
-10 & -6 & 12 & 0 \\
10 & -2 &4 & 0 \\
0 & -22 & 9 & 0 \\
1 & 5 & -10 & 0 \\
0 & -12 & 12 & -14 \\
0 & -6 & 9 & -14 \\
0 & 0 & 0 & 1
\end{pmatrix} \ .
$$
Of course we have $BA=O$.

The corresponding coboundary maps in
$$
C^6(\mathfrak{ham}_2,\mathfrak{sp}(2,\bR))_{8}
\overset{\widetilde{A}}{\lra}
C^7(\mathfrak{ham}_2,\mathfrak{sp}(2,\bR))_{8}
\overset{\widetilde{B}}{\lra}
C^8(\mathfrak{ham}_2,\mathfrak{sp}(2,\bR))_{8}
$$
are given by the following $(14,18)$-matrix $\widetilde{A}$
and $(4,14)$-matrix $\widetilde{B}$:
$$
\widetilde{A}=
\begin{pmatrix}
A & A_1\\
O & A_2
\end{pmatrix},
\quad
\widetilde{B}=
\begin{pmatrix}
B \\
B_1
\end{pmatrix}.
$$
Here $O$ denotes the zero matrix of size $(2,9)$ and
this checks the fact that
$$
C^*_{GF}(\mathfrak{ham}_2^0,\mathfrak{sp}(2,\bR))_{10}
\ \subset\
C^{*+2}_{GF}(\mathfrak{ham}_2,\mathfrak{sp}(2,\bR))_{8}
$$
is indeed a subcomplex.
Now an explicit computation shows that the above inclusion
induces an isomorphism in cohomology.
This completes the proof of Theorem \ref{th:main}. 
\qed

\begin{remark}
The unique leaf cohomology class
$$
\eta\in H^5_{GF}(\mathfrak{ham}_2^0,\mathfrak{sp}(2,\bR))_{10}
$$
such that $\eta\wedge\om=GKF$
can be represented by an explicit cocycle in
$C^5_{GF}(\mathfrak{ham}_2^0,\mathfrak{sp}(2,\bR))_{10}$
which is a linear combination of the 
cochains of the forms
$(3^347), (3^24^26), (3^245^2)$
in Table \ref{tab:4}.
We omit the precise formula.
\end{remark}

\begin{proof}[Proof of Corollary \ref{cor:nt}]
We proved in \cite{KM03} (see also \cite{KM07})
that both the first Pontrjagin
class $p_1\in H^4(\mathrm{ESymp}^\delta(\Sigma_g);\bR)$ 
and its fiber integral
$e_1\in H^2(\mathrm{BSymp}^\delta(\Sigma_g);\bR)$,
which is the first Mumford-Morita-Miller class, are non-trivial.
More precisely, we proved the existence of 
foliated $\Sg$-bundles over closed oriented surfaces 
such that the signatures of their total spaces are
non-zero, while their total holonomy groups are
contained in the group $\mathrm{Symp}(\Sg)$ of
area-preserving diffeomorphisms of $\Sg$ (with respect to some area form).
By Theorem \ref{th:main} the homomorphism
$$
\wedge\omega\colon
H^2_{GF}(\mathfrak{ham}_2^0,\mathfrak{sp}(2,\bR))\cong\bR
\lra
H^4_{GF}(\mathfrak{ham}_2,\mathfrak{sp}(2,\bR))\cong\bR
$$
is an isomorphism, where the target
is generated by the first Pontrjagin class $p_1$.
The result follows.
\end{proof}

\begin{proof}[Proof of Theorem \ref{th:nt2}]
We begin with the proof of the first statement.
On the one hand, the weight of the elements in $S^3 H_\bR^{2n}$ is $1$
while that of $\om^n$ is $-2n$. Hence the weights of 
elements of $\mathrm{Im}\Phi$ restricted to the range
$*\leq 2n$ are non-positive. 
By the result of Gel'fand-Kalinin-Fuks \cite{GKF} mentioned
in the Introduction, we can conclude that
$\mathrm{Im}\Phi$ is contained in the span of the
classes
$$
\omega^k p_1^{k_1}\cdots p_n^{k_n} \in 
H^*_{GF}(\mathfrak{ham}_{2n},\mathrm{Sp}(2n,\bR))
$$
with $k+k_1+2k_2\cdots +n k_n \leq n$.
On the other hand, any element in $\mathrm{Im}\Phi$ is 
annihilated by taking the wedge product with a single $\om$
because $\om^{n+1}$ vanishes identically. Therefore
$\mathrm{Im}\Phi$ is contained in the span of the above classes 
with the condition that
$k+k_1+2k_2\cdots +n k_n$ is precisely equal to $n$.

It remains to prove that all these classes are indeed
contained in $\mathrm{Im}\Phi$.
For this, we use the well-known formulae which 
express $\om$ and Pontrjagin classes in terms of the
tautological $1$-forms.
With respect to the standard basis
$x_1,\cdots,x_n,y_1,\cdots,y_n$ of the symplectic vector
space $H^{2n}_\bR$ and the tautological $1$-forms
$$
\delta^{i}_{j_1\cdots j_k}\in C^1_{GF}(\mathfrak{ham}_{n}),
$$
we have 
$$
\omega=\delta^1\wedge\delta^{n+1}+\cdots+\delta^{n}\wedge\delta^{2n}.
$$
The universal curvature form $\Omega=(\Omega^i_j)$
can be written as
$$
\Omega^i_j=\sum_{k=1}^n \delta^i\wedge \delta^k_{jk} \ ,
$$
see e.g. \cite{B}, and the Pontrjagin classes $p_i\ (i=1,2,\cdots)$
are certain homogeneous polynomials
on $\Omega^i_j$ of degree $2i$. In terms of the duals of
Hamiltonian functions,
the tautological forms $\delta^i$ and $\delta^k_{jk}$ correspond
to elements of $H^*$ and $S^3 H^*$, respectively. We can now conclude that
$$
p_i\in (\La^{2i} H^*\otimes \La^{2i} S^3 H^*)^{\mathrm{Sp}(2n,\bR)}
\quad (i=1,2,\cdots).
$$
It follows that any element
$\omega^k p_1^{k_1}\cdots p_n^{k_n}$
with $k+k_1+2k_2\cdots +n k_n = n$
is contained in
$$
(\La^{2n} H^*\otimes \La^{2n-2k} S^3 H^*)^{\mathrm{Sp}(2n,\bR)}
\cong 
\om^n\wedge (\La^{2n-2k} S^3 H^*)^{\mathrm{Sp}(2n,\bR)}
$$
because $\La^{2n} H^*\cong\bR$ generated by $\om^n$.
Hence such elements are contained in $\mathrm{Im}\Phi$
proving the first part of the Theorem.

Next we prove the second part.
Let $\pi\colon E\lra X$ be a foliated $\Sg$-bundle over a closed
oriented surface with non-vanishing signature such that
the total holonomy group is contained in the group
$\mathrm{Symp}(\Sg)$. The existence of such bundles was
proved in our paper \cite{KM03}. The classifying map
$f\colon E\lra \mathrm{B\Gamma}_2^\om$ of the transversely 
symplectic foliation on $E$ of codimension $2$ has the property that
$f^*(p_1)\neq 0$. Now consider the manifold
$\bC P^{n-k}\times E^k$ equipped with transversely symplectic
foliation of codimension $2n$ which is induced from the 
point foliation on $\bC P^{n-k}$ and the above foliation on $E$.
Then it is easy to see that the characteristic class 
$\om^{n-k}p_1^k$ of this foliation is non-trivial.
This completes the proof.
\end{proof}

\begin{remark}
By the calculation in the proof $p_1$ is divisible by $\omega$ only if $n=1$, although
in general $p_1^n$ is divisible by $\omega$.
The dimension of $(\La^{2i} H^*\otimes \La^{2i} S^3 H^*)^{\mathrm{Sp}(2n,\bR)}$
is $1$ for $n=1$ and is $2$ for $n\geq 2$.
\end{remark}

\section{The Euler characteristic of 
$H^*(\mathfrak{ham}^0_{2n},\mathrm{Sp}(2n,\bR))$}

As we mentioned already,
Perchik \cite{P} gave a formula for the generating function
$$
\sum_{w=0}^\infty \chi(H^*(\mathfrak{ham}_{2n},\mathrm{Sp}(2n,\bR))_w)t^w
$$
of the Euler characteristic of the relative cohomology of
$\mathfrak{ham}_{2n}$. In this section, we prove a similar
formula for 
$H^*_{GF}(\mathfrak{ham}^0_{2n},\mathrm{Sp}(2n,\bR))$.

Following \cite{P}, let us define rational functions 
$p_i(n)\ (i=0,1,\cdots)$
on $n+1$ variables $x_1,\cdots,x_n, t$ 
(polynomials with respect to $t$) as follows.
First consider
$$
a=(a_1,\cdots,a_n),\quad b=(b_1,\cdots,b_n)\quad (a_i,b_i\geq 0)
$$
and put
$$
|a+b|=\sum (a_i+b_i),\quad x^{a-b}=x_1^{a_1-b_1}\cdots x_n^{a_n-b_n}.
$$
Then define
\begin{align*}
p_0(n)&=\Pi_{|a+b|=2,\ a\not= b} (1-x^{a-b})\\
p_k(n)&=\Pi_{|a+b|=2+k} (1-t^{k}x^{a-b}) \ .
\end{align*}

\begin{theorem}
The constant term with respect to $x_i$ $(i=1,\cdots,n)$ of the infinite product
$\Pi_{i=0}^\infty\ p_i(n)$ is equal to 
$$
n! 2^n \sum_{w=0}^\infty 
\chi(H^*(\mathfrak{ham}^0_{2n},\mathrm{Sp}(2n,\bR))_w)t^w \ ,
$$
where the subscript $w$ denotes the weight $w$ part of the cohomology.
\label{th:gf}
\end{theorem}

\begin{proof}
Perchik's formula was obtained by multiplying
the above infinite product
$\Pi_{i=0}^\infty\ p_i(n)$ with one more rational function
$$
p_{-1}(n)=\Pi_{|a+b|=1,\ a\not= b} (1-t^{-1}x^{a-b}).
$$
This part corresponds to the constant term
of $\mathfrak{ham}_{2n}$ which is isomorphic
to $H^{2n}_\bR$ as a representation of $\mathrm{Sp}(2n,\bR)$
and whose weight is $-1$.
Since the relative cohomology
$H^*_{GF}(\mathfrak{ham}^0_{2n},\mathrm{Sp}(2n,\bR))$
is defined by ignoring this part,
the proof follows by eliminating $p_{-1}(n)$
from the original formula of Perchik.
\end{proof}

\begin{remark}
In the case of $n=1$, a computer computation carried out with the help of 
M.~Suzuki shows that $\frac{1}{2}$ times the above
constant term in low degrees in $t$ is given by
$$
1+ t^2-t^{10}+ t^{12}- t^{14}- t^{16}+ t^{18}-3 t^{24}+2 t^{26}+\cdots \ ,
$$
while the corresponding series for 
$H^*_{GF}(\mathfrak{ham}^0_2,\mathfrak{sp}(2,\bR))$ due to Perchik
is
$$
t^{-2}+2-t^{8}- t^{14}- t^{22}- t^{28}+ t^{30}- t^{32}+\cdots \ .
$$
The coefficient $-1$ of $t^{10}$ in the former series corresponds
to our leaf cohomology class $\eta$.
Observe also that the coefficient of $t^{16}$ is
$-1$. The corresponding coefficient of $t^{14}$ in the latter series
is also $-1$ which represents the Metoki class in 
$H^{9}_{GF}(\mathfrak{ham}_2,\mathfrak{sp}(2,\bR))_{14}$
(see \cite{Metoki}). Although the cocycle given by Metoki himself
is not divisible by $\omega$, it seems highly likely that
his class can also be decomposed as $\eta'\wedge\omega$
for some leaf cohomology class
$\eta'\in H^7_{GF}(\mathfrak{ham}^0_2,\mathfrak{sp}(2,\bR);\bR)_{16}$.
\end{remark}

\section{Concluding remarks}

It is easy to see that the relative cohomology 
$$
H^*_{GF}(\mathfrak{ham}^0_{2n},\mathrm{Sp}(2n,\bR))_w
\cong
H^*_{GF}(\mathfrak{ham}^1_{2n})_w^{\mathrm{Sp}(2n,\bR)}
$$
stabilizes as $n$ goes to infinity. In fact, the limit
cohomology is nothing but one of Kontsevich's 
theories of graph cohomologies developed in
\cite{Kontsevich93}, \cite{Kontsevich94},
more precisely the {\it commutative} case
(see \cite{K}).
As before, the abelianization homomorphism
$$
\mathfrak{ham}_{2n}^1\lra S^3 H^{2n}_\bR
$$
induces a homomorphism
$$
\Phi_n\colon H^*(S^3 H^{2n}_\bR)^{\mathrm{Sp}(2n,\bR)}
\lra H^*_{GF}(\mathfrak{ham}^1_{2n})^{\mathrm{Sp}(2n,\bR)}.
$$
However, the stable cohomology
$$
\lim_{n\to\infty} H^*(S^3 H^{2n}_\bR)^{\mathrm{Sp}(2n,\bR)}
$$
is isomorphic to
$$
\bR[\text{vertex oriented connected trivalent graph}]/(\text{AS}) \ ,
$$
where $AS$ denotes the anti-symmetric relation.
If we add another relation, called the $IHX$ relation,
to the above, we obtain the algebra
$$
\mathcal{A}(\phi)=
\bR[\text{vertex oriented connected trivalent graph}]/(\text{AS, IHX}).
$$
This algebra plays a fundamental role in the theory of
finite type invariants for homology $3$-spheres 
due to Ohtsuki \cite{O}, who extended the foundational theory of Vassiliev
for knots, and developed by
Le, Murakami and Ohtsuki (see \cite{LMO}).

Garoufalidis and Nakamura proved the following result:
\begin{theorem}[{\bf Garoufalidis and Nakamura \cite{GN}}]
The ideal $(S^4 H^{2n}_\bR)$ of $\Lambda^* S^3H^{2n}_\bR$
generated by $S^4 H^{2n}_\bR\subset \Lambda^2 S^3 H^{2n}_\bR$
corresponds exactly to the $IHX$-relation so that
there is an isomorphism
$$
\mathcal{A}(\phi)\cong
(\Lambda^* S^3H^{2n}_\bR/(S^4 H^{2n}_\bR))^{\mathrm{Sp}(2n,\bR)}.
$$
\end{theorem}
Since it can be seen that $\mathrm{Ker}\ \Phi_\infty$
coincides with the $\mathrm{Sp}$-invariant part
$\left((S^4 H^{2n}_\bR)\right)^{\mathrm{Sp}(2n,\bR)}$
of the above ideal, we conclude that
$$
\mathrm{Image}\ \Phi_\infty\cong
\mathcal{A}(\phi) \ .
$$
Thus it is a very important problem to determine
$\mathrm{Coker}\ \Phi_\infty$. We have tried to determine
whether our leaf cohomology class
$\eta\in H^5_{GF}(\mathfrak{ham}_2^0,\mathfrak{sp}(2,\bR))_{10}$
survives in the {\it stable} cohomology
$$
\lim_{n\to\infty} H^5_{GF}(\mathfrak{ham}^0_{2n},\mathfrak{sp}(2n,\bR))_{10} \ ,
$$
or not. We have the same problem for other unstable
leaf cohomology classes.
So far this attempt has remained unsuccessful. 
One method of attacking this problem would be to
compute the generating function $c(t)$ for the
commutative graph cohomology by making use of
Theorem~\ref{th:gf}. More precisely,
there is important problem of computing
$$
c(t)=
\lim_{n\to\infty}
\sum_{w=0}^\infty 
\chi(H^*(\mathfrak{ham}^0_{2n},\mathrm{Sp}(2n,\bR))_w)t^w \ ,
$$
which is the limit as $n\to\infty$ of the
formula given in Theorem~\ref{th:gf}.

Our computations so far imply
$$
c(t)=1+t^2+2t^4+3t^6+6t^8+\cdots \ .
$$
Recall here that the algebra $\mathcal{A}(\phi)$ is known
(see \cite{O2}) to be
a polynomial algebra whose numbers of generators
are $1, 1, 1, 2, 2, 3, \cdots $ in degrees $2,4,6,8,10,12,\cdots$
so that the generating function for this algebra is
$$
1+t^2+2t^4+3t^6+6t^8+9t^{10}+16 t^{12}+\cdots.
$$
It should be nice to know how these two generating
functions differ from each other.

\subsection*{Acknowledgements}
The authors would like to thank T.~Sakasai and M.~Suzuki
for help with computer computations using LiE and Mathematica. 
They would also like to thank T.~Tsuboi for information
about the thesis of Metoki \cite{Metoki}.

\bibliographystyle{amsplain}

\end{document}